\documentclass[twoside,12pt]{article}
\usepackage{color}

\usepackage{graphics}
\usepackage{graphicx}

\usepackage{amssymb,amsfonts}
\usepackage{amsmath}
\usepackage{amsthm}

\usepackage[table]{xcolor}

\usepackage[all]{xy}
\usepackage{tikz}
\usetikzlibrary{arrows,chains,matrix,positioning,scopes}

\tikzset{>=stealth',every on chain/.append style={join},
         every join/.style={->}}
\tikzstyle{labeled}=[execute at begin node=$\scriptstyle,
   execute at end node=$]
   
\usepackage{forest}
  
\tikzset{%
   highlight/.style={rectangle,rounded corners,fill=blue!15,draw,fill opacity=0.5
   ,thick,inner sep=-1pt, 
   }
 }

\def\hpt45{${\cal HPT}_{\!45}$}

\def\binomh#1#2{ \scalebox{.3}[1.2]{\textbf{)}}{\genfrac{}{}{0pt}{}{#1}{#2}}\scalebox{.3}[1.2]{\textbf{(}} }

\def\binomhh#1#2{ \scalebox{.4}[1.7]{\textbf{)}}{\genfrac{}{}{0pt}{}{#1}{#2}}\scalebox{.4}[1.7]{\textbf{(}} }

\newtheorem{theorem}{Theorem}
\newtheorem{lemma}{Lemma}[section]
\newtheorem{cor}{Corollary}
\newtheorem{rem}{Remark}

\title{\bf Recurrence sequences in the hyperbolic Pascal triangle corresponding to the regular mosaic $\{4,5\}$ 
}

\author{L\'aszl\'o N\'emeth\footnote{Institute of Mathematics, University of West Hungary,   Sopron. \textit{nemeth.laszlo@nyme.hu}}, 
	L\'aszl\'o Szalay \footnote{Institute of Mathematics,University of West Hungary,   Sopron.  \textit{szalay.laszlo@nyme.hu}} \footnote{Department of Mathematics and Informatics, J. Selye University, Hradna ul. 21., 94501 Komarno, Slovakia.}}

\date{}

\begin{document}

\maketitle

\begin{abstract}
Recently, a new generalization of Pascal's triangle, the  so-called hyperbolic Pascal triangles were introduced. The mathematical background goes back to the regular mosaics in the hyperbolic plane. 
In this article, we investigate the paths in the hyperbolic Pascal triangle corresponding to the regular mosaic $\{4,5\}$, in which the binary recursive sequences $f_{n}=\alpha f_{n-1}\pm f_{n-2}$ are represented ($\alpha\in\mathbb{N}^+$). 
\smallskip

\end{abstract}

\section{Introduction}\label{sec:introduction}
 
In the hyperbolic plane there are an infinite number of types of regular mosaics (see, for example \cite{C}), they are assigned by Schl\"afli's symbol $\{p,q\}$, where the positive integers $p$ and $q$  satisfy $(p-2)(q-2)>4$. Each regular mosaic induces a so-called hyperbolic Pascal triangle (see \cite{BNSz}), following and generalizing the connection between the classical Pascal's triangle and the Euclidean regular square mosaic $\{4,4\}$.
For more details see \cite{BNSz}, but here we also collect some necessary information. 

There are several approaches to generalize Pascal's arithmetic triangle (see, for instance \cite{BSz}).
The hyperbolic Pascal triangle based on the mosaic $\{p,q\}$ can be figured as a digraph, where the vertices and the edges are the vertices and the edges of a well defined part of the lattice $\{p,q\}$, respectively, further each vertex possesses a value, say label, giving the number of different shortest paths from the fixed base vertex. Figure~\ref{fig:Pascal_layer6} illustrates the hyperbolic Pascal triangle linked to $\{p,q\}=\{4,5\}$. 
Generally, for $\{4,q\}$, the quadrilateral shape cells surrounded by appropriate edges are corresponding to the squares in the mosaic. The base vertex has two edges (both are outgoing), the leftmost and the rightmost vertices have three (one ingoing and two outgoing), the others have $q$ edges (either two ingoing and $q-2$ outgoing (type $A$) or one ingoing and $q-1$ outgoing (type $B$)). In other words, apart from the winger elements, vertices of type $A$ have two ascendants and $q-2$ descendants, vertices of type $B$ do one ascendant and $q-1$ descendants. In the figures, we denote the $A$-type vertices by red circle and $B$-type vertices by cyan diamond, further the wingers by white diamond. 
The vertices having distance $n$ from the base vertex are located in row $n$. 
The general method of drawing is the following. Going along the vertices of the $n^{th}$ row, according to type of the elements (winger, $A$, $B$), we draw appropriate number of edges downward (2, $q-2$, $q-1$, respectively). Neighbor edges of two neighbor vertices of the $n^{th}$ row meet in the $(n+1)^{th}$ row, constructing a vertex of type $A$. The other descendants of row $n$ in row $n+1$ have type $B$, except the two wingers.
In the sequel, $\binomh{n}{k}$ denotes the $k^\text{th}$ element in row $n$, which is either the sum of the labels of its two ascendants or coincide the label of its unique ascendant. For instance, if $\{p,q\}=\{4,5\}$, then
$$\binomhh{4}{6}=5=2+3=\binomhh{3}{2}+\binomhh{3}{3} \qquad {\rm and} \qquad \binomhh{4}{5}=2=\binomhh{3}{2}$$ 
hold (see Figure~\ref{fig:Pascal_layer6}). We note, that the hyperbolic Pascal triangle has the property of vertical symmetry.

\begin{figure}[h!]
 \centering
 \includegraphics[width=0.99\linewidth]{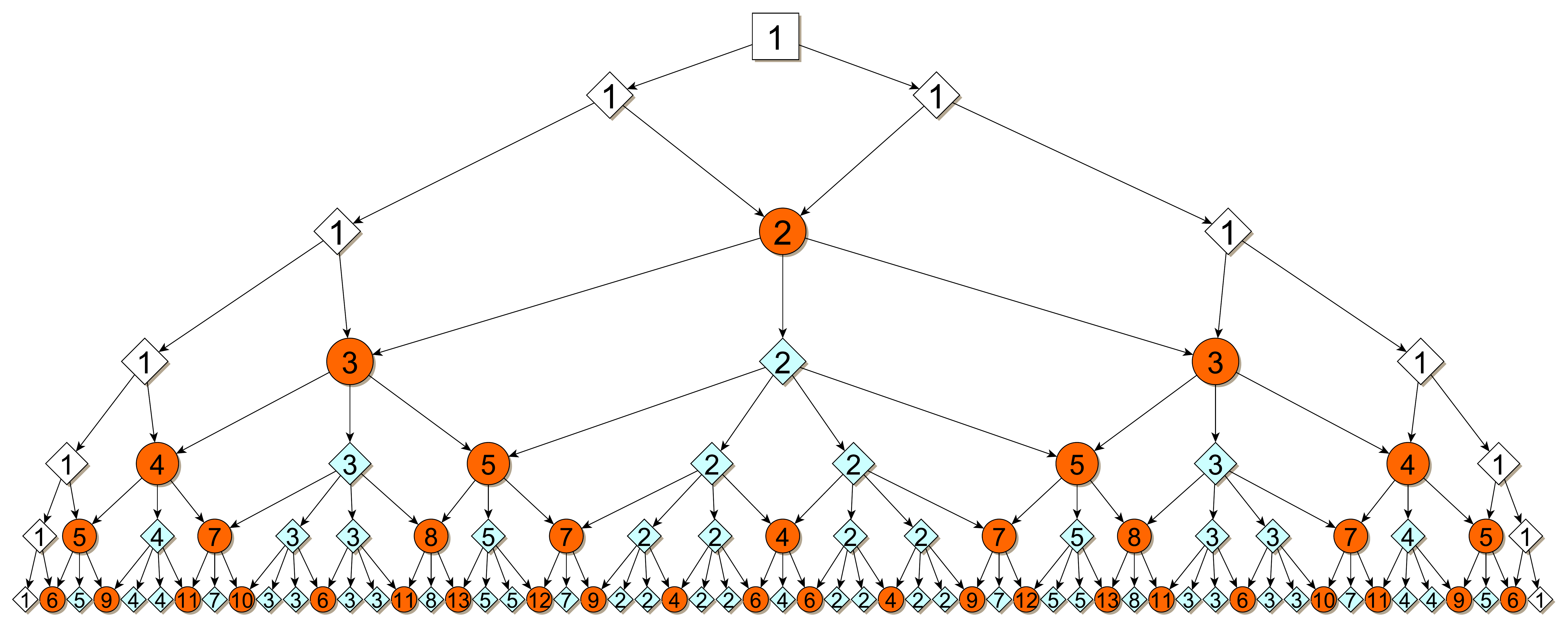}
 \caption{Hyperbolic Pascal triangle linked to $\{4,5\}$ up to row 6}
 \label{fig:Pascal_layer6}
\end{figure}

\section{Recurrence sequences linked to $\{4,5\}$}

Let $\{p,q\}=\{4,5\}$ be fixed, further we let \hpt45  denote the hyperbolic Pascal triangle corres\-ponding to the mosaic $\{4,5\}$. It was showed in \cite{BNSz} that all the  binary recurrence sequences $(f_i)_{i\geq0}$  which are defined by 
\begin{equation} \label{seq:+}
f_{i}=\eta f_{i-1}+f_{i-2}, \quad (n\geq 2),
\end{equation}
where $\eta$ and  $f_0<f_1$ are positive integers, appear in \hpt45.

In the following we describe paths corresponding to further positive integer binary recurrence sequences. We remark that although we restrict ourselves to \hpt45, the methods and the results have been worked out can be fitted to other hyperbolic Pascal triangles with $p=4,~q\ge6$.

Taking a vertex of type $A$ in row $n$, it has exactly two descendants of type $A$ in the row $n+1$. In order to reach and distinguish them, we denote the left-down step and right-down step (along the appropriate edge of the graph) by $L$ and $R$, respectively. For the sake of brevity, the sequence of $\ell+r$ consecutive steps
\begin{equation*} 
\underbrace{LL\cdots L}_{\ell}\underbrace{RR\cdots R}_{r} 
\end{equation*}
will be denoted by $L^\ell R^r$. Till the end of this work, such a path is always considered on vertices of type $A$. Generally, we are interested in the labels of these vertices, therefore sometimes we call them elements (as the elements or terms of a sequence), but if it is necessary we determine the location of the element, too. 

This paper will use the next theorem (Theorem 5 in \cite{BNSz}), which states that any two positive integers can be found next to each other somewhere in \hpt45.

\begin{theorem}\label{thm:uv}
Given $u,v\in \mathbb{N}^+$, then there exist $n,k\in \mathbb{N}^+$ such that $u=\binomh{n}{k}$ and $v=\binomh{n}{k+1}$. 
\end{theorem}

Using Theorem~\ref{thm:uv}, Corollary~\ref{cor:uv} provides an immediate consequence of the properties of \hpt45. 

\begin{cor}\label{cor:uv}
If $u=\binomh{n}{k}<v=\binomh{n}{k+1}$ holds for some positive integers $u$ and $v$, then $\binomh{n}{k+2}=v-u$, moreover the type of $\binomh{n}{k+1}$ is $A$,  while the types of $\binomh{n}{k}$ and $\binomh{n}{k+2}$ are not $A$ (i.e.,~either $B$ or winger). 
\end{cor}

\begin{rem}\label{rem1} Clearly, by the symmetry we also have the construction $u=\binomh{n}{k}>v=\binomh{n}{k+1}$ and $\binomh{n}{k-1}=u-v$. Further, the type of $\binomh{n}{k}$ is $A$.
\end{rem} 

\subsection{Recurrence sequences and paths}

Let $(f_i)_{i\geq0}$ be a recurrence sequence defined by
\begin{equation}\label{seq:-}
f_{i}=\alpha f_{i-1}- f_{i-2},\qquad (i\geq2),
\end{equation} 
where $\alpha\in \mathbb{N}^+$, $\alpha \geq2$,  and  $f_0<f_1$ are positive integers with $\gcd(f_0,f_1)=1$. If $\alpha=2$ then $(f_i)_{i\geq0}$  is an arithmetic progression given by $f_i=f_{i-1}+(f_1-f_0)$.  

From Theorem \ref{thm:uv} and Corollary \ref{cor:uv} we know that in case of any positive integers $f_0<f_1$, there exist an element in \hpt45 with value $f_1$, and with neighbors in the same row valued by $f_0$ and $f_1-f_0$. 
In Theorem~\ref{thm:seq01} we give a  path in \hpt45 (analogously to Theorem 6 in \cite{BNSz}) contains all the elements of \eqref{seq:-}. 

\begin{theorem}\label{thm:seq01}
There exists a path in \hpt45 crossing vertices of type $A$, such that the vertices are labelled with the terms of $(f_i)_{i\geq1}$ as follows. Assume that $\binomh{n}{k}=f_1$, and $\binomh{n}{k-1}=f_1-f_0$. Then the first element of the path is $f_1$ and the pattern of the steps from $f_{i-1}$ to $f_i$ $(i\geq2)$ is 
$LR^{\alpha-2}$.
\end{theorem}
\begin{proof}
According to Theorem \ref{thm:uv}, any $f_1$ and $f_1-f_0$ can be neighbours in \hpt45, where type of $f_1$ is $A$ (and the type of $f_1-f_0$ is not $A$).

If $\alpha=2$, then the statment is easy to show, since no $R$ steps. Indeed, the difference of an element type $A$, and its immediate left descendant having type $A$ is the constant $f_1-f_0$. 

Assume now $\alpha\ge3$.
 By the construction rule of \hpt45, we can follow the way from any $f_{i-1}$ to $f_i$ $(i\geq2)$ in Figure~\ref{fig:steps01-f}, which justifies the theorem (the type of the rectangle shaped elements is $A$). In the last row of the figure we use, among others,
the equality $f_i-f_{i-1}=(\alpha-1)f_{i-1}-f_{i-2}$.
\end{proof}

\begin{figure}[h!] \centering
\scalebox{0.99}{ \begin{tikzpicture}[->,xscale=1.5,yscale=1.5, auto,swap]
  \node(a0) at (0,0)    {$f_{i-1}-f_{i-2}$};
  \node (a1) at (0.5,-1) [shape=rectangle,draw,fill=red!5] {$2f_{i-1}-f_{i-2}$};
  \node (a2) at (1,-2)  [shape=rectangle,draw,fill=red!5] {$3f_{i-1}-f_{i-2}$};
  \node (a3) at (1.5,-3) {$\vdots $};
  \node (a4) at (2,-4)  [shape=rectangle,draw,fill=red!5] {$(\alpha-1)f_{i-1}-f_{i-2}$};
  \node (a5) at (2.5,-5)  [shape=rectangle,draw,fill=red!30] {$f_{i}=\alpha f_{i-1}-f_{i-2}$};
  \node (b0) at (2,0)   [shape=rectangle,draw,fill=red!30] {$f_{i-1}$};
  \node (b1) at (2.5,-1) {$f_{i-1}$};
  \node (b2) at (3,-2)   {$f_{i-1}$};
  \node (b3) at (3.5,-3) {$\vdots$};
  \node (b4) at (4,-4) {$f_{i-1}$};
  \node (b5) at (4.5,-5) {$f_{i-1}$};
  \node (c0) at (0,-5) {$f_{i}-f_{i-1}$};
  \node (d0) at (4,0) {$f_{i-2}$};
     
  \path (a0) edge node {} (a1) ;
  \path (a1) edge node {} (a2) [very thick];
  \path (a2) edge node {} (a3) [very thick];
  \path (a3) edge node {} (a4) [very thick];
  \path (a4) edge node {} (a5) [very thick];
  \path (b0) edge node {} (b1);
  \path (b1) edge node {} (b2);
  \path (b2) edge node {} (b3);
  \path (b3) edge node {} (b4);
  \path (b4) edge node {} (b5);
  \path (a4) edge node {} (c0);
  \path (b0) edge node {} (a1) [very thick];   
  \path (b1) edge node {} (a2);                       
  \path (b2) edge node {} (a3);   
  \path (b3) edge node {} (a4);  
  \path (b4) edge node {} (a5);  
\end{tikzpicture}}
  \caption{Path $LR^{\alpha-2}$ between $f_{i-1}$ and $f_i$}
  \label{fig:steps01-f}
\end{figure}
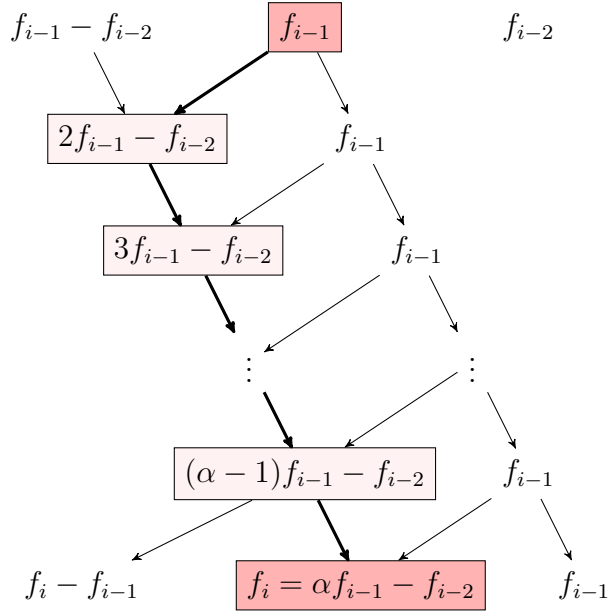

\begin{rem}
Theorem \ref{thm:seq01} can be extended for the whole sequence $(f_i)_{i\geq0}$  if and only if $(\alpha-1)f_0<f_1<\alpha f_0$. Under these conditions one can follow the path back from the bottom of Figure~\ref{fig:steps01-f} to the top, from $f_1$ to $f_0$.
\end{rem}

The path showed on the right hand side of Figure~\ref{fig:Pascal_4f-f} (cf.~Figure \ref{fig:Pascal_layer6}) is an example for the binary recurrence $f_{i}=4f_{i-1}-f_{i-2}$ with $f_0=1$, $f_1=2$. 

Theorem~\ref{thm:seq01} finds a path to the sequence~\eqref{seq:-}. Considering the opposite direction, now we decribe the sequence corresponding to a given pattern of steps. The expression ``corner element'' means a labelled vertex where the direction of the sequence of steps changes. For example, the first corner element of the path $L^3R^2$ is the vertex reached after three left steps, the second corner element comes after further two right steps, etc.

\begin{theorem}\label{th:lr}
Suppose that the A-type vertex $\binomh{n}{k}=U_1=u_1$ is a starting point of the path $L^\ell R^r$. We let $U_i$, and $u_i$ ($i=1,\dots$) denote the label of the corner elements, and the label of every second corner elements of the path, respectively. Then we have   
\begin{equation} \label{eq:lr}
u_{i}=(\ell r+2)u_{i-1}-u_{i-2},\qquad\qquad (i\ge3).
\end{equation}
Moreover, if $\ell=r$, then  
\begin{equation*} 
U_{i}=\ell U_{i-1}+U_{i-2},\qquad\qquad (i\ge3).
\end{equation*}
\end{theorem}

Obviously, $u_i=U_{2i-1}$ holds. The proof of Theorem \ref{th:lr} applies the following lemma (see \cite{BNSz}, Remark 1 linked to Lemma 4).

\begin{lemma}\label{lemma:2seq}
Let $x_0$, $y_0$, further $a_j$ and $b_j$ ($j=1,2$) be complex numbers such that $a_2b_1\ne0$. Assume that
for $i\ge i_0$  the terms of the sequences $(x_i)$ and $(y_i)$ satisfy
\begin{eqnarray*}
x_{i+1}&=&a_1x_i+b_1y_i,\\
y_{i+1}&=&a_2x_i+b_2y_i.
\end{eqnarray*}
Then for both sequences
\begin{equation*}\label{nomix}
z_{i+2}=(a_1+b_2)z_{i+1}+(-a_1b_2+a_2b_1)z_{i}
\end{equation*}
holds ($i\ge i_0$).
\end{lemma}

\begin{proof}[Proof of Theorem~\ref{th:lr}]
Suppose that $v_1$ is the left ascendant of $u_1$.
By Figure \ref{fig:stepsleft-right}, which demonstrates the path precisely from $u_i$ to $u_{i+2}$ $(i\geq1)$ along vertices type $A$ in \hpt45, we gain the system of the recursive equations 
\begin{eqnarray}
u_{i+1} &=& (r+1)u_i+\left(\ell+r(\ell-1)\right)v_i,\label{eq:ui}\\
v_{i+1} &=& ru_i+\left(\ell+(r-1)(\ell-1)\right)v_i. \nonumber
\end{eqnarray}

Using Lemma \ref{lemma:2seq}  we receive that both $u_i$ and $v_i$ satisfy the equation  
\begin{equation*}
z_{i+2}=(\ell r+2) z_{i+1}-z_i.
\end{equation*}
If $\ell=r$, then we simply obtain
\begin{eqnarray}
U_{i+1} &=& U_i+\ell V_i,\label{eq:Ui}\\
V_{i+1} &=& U_i+(\ell-1)V_i.\nonumber
\end{eqnarray}
Now Lemma \ref{lemma:2seq} results  that $U_i$ and $V_i$ satisfy the equation
\begin{equation*}\label{eq:Xl}
Z_{i+2}=\ell Z_{i+1}+Z_i.
\end{equation*}
\end{proof}

\begin{figure}[h!] \centering
\scalebox{0.99}{ \begin{tikzpicture}[->,xscale=1.5,yscale=0.92, auto,swap]
  \node(a0) at (0,0)    {$\vdots $};
  \node (a1) at (0.5,-1)  [shape=rectangle,draw]{$v_i$};
  \node (a2) at (1,-2)  [shape=rectangle,draw,fill=red!30] {$u_i$};
  \node (a3) at (0.5,-3) {$u_i+v_i$};
  \node (a4) at (0,-4)   {$u_i+2v_i$};
  \node (a5) at (-0.5,-5)  {$\vdots $};
  \node (a6) at (-1,-6) [shape=rectangle,draw] {$V_j=u_i+(\ell-1)v_i$}; 
  \node (a7) at (-1.5,-7)  [shape=rectangle,draw,fill=red!30]{$U_j=u_i+\ell v_i$};  
  \node (a8) at (-1,-8) {$U_j+V_j$};
  \node (a9) at (-0.5,-9) {$U_j+2V_j$};   
  \node (a10) at (0,-10)  {$\vdots $};
  \node (a11) at (0.5,-11)  [shape=rectangle,draw] {$v_{i+1}=V_{j+1}=U_j+(r-1)V_j$}; 
  \node (a12) at (1,-12)  [shape=rectangle,draw,fill=red!30]{$u_{i+1}=U_{j+1}=U_j+r V_j$}; 
  
   \node (a13) at (0.5,-13) {$u_{i+1}+v_{i+1}$};
   \node (a14) at (0,-14)   {$u_{i+1}+2v_{i+1}$};
   \node (a15) at (-0.5,-15)  {$\vdots $};
   \node (a16) at (-1,-16) [shape=rectangle,draw] {$V_{j+2}=u_{i+1}+(\ell-1)v_{i+1}$}; 
   \node (a17) at (-1.5,-17)  [shape=rectangle,draw,fill=red!30]{$U_{j+2}=u_{i+1}+\ell v_{i+1}$};  
   \node (a18) at (-1,-18) {$U_{j+2}+V_{j+2}$};
   \node (a19) at (-0.5,-19) {$U_{j+2}+2V_{j+2}$};   
   \node (a20) at (0,-20)  {$\vdots $};
   \node (a21) at (0.5,-21) [shape=rectangle,draw] {$v_{i+2}=V_{j+3}=U_{j+2}+(r-1)V_{j+2}$}; 
   \node (a22) at (1,-22)  [shape=rectangle,draw,fill=red!30]{$u_{i+2}=U_{j+3}=U_{j+2}+r V_{j+2}$};   
         
  \path (a0) edge node {} (a1) [very thick];
  \path (a1) edge node {} (a2) [very thick];
  \path (a2) edge node {} (a3) [very thick];
  \path (a3) edge node {} (a4) [very thick];
  \path (a4) edge node {} (a5) [very thick];
  \path (a5) edge node {} (a6) [very thick];
  \path (a6) edge node {} (a7) [very thick];
  \path (a7) edge node {} (a8) [very thick];
  \path (a8) edge node {} (a9) [very thick];
  \path (a9) edge node {} (a10) [very thick];
  \path (a10) edge node {} (a11) [very thick];
  \path (a11) edge node {} (a12) [very thick];
  \path (a12) edge node {} (a13) [very thick];  
    \path (a13) edge node {} (a14) [very thick];
    \path (a14) edge node {} (a15) [very thick];
    \path (a15) edge node {} (a16) [very thick];
    \path (a16) edge node {} (a17) [very thick];
    \path (a17) edge node {} (a18) [very thick];
    \path (a18) edge node {} (a19) [very thick];
    \path (a19) edge node {} (a20) [very thick];
    \path (a20) edge node {} (a21) [very thick];
    \path (a21) edge node {} (a22) [very thick];
  \end{tikzpicture}}
  \caption{Path $L^{\ell}R^{r}$ from $u_{i}$ to $u_{i+2}$}
  \label{fig:stepsleft-right}
\end{figure}
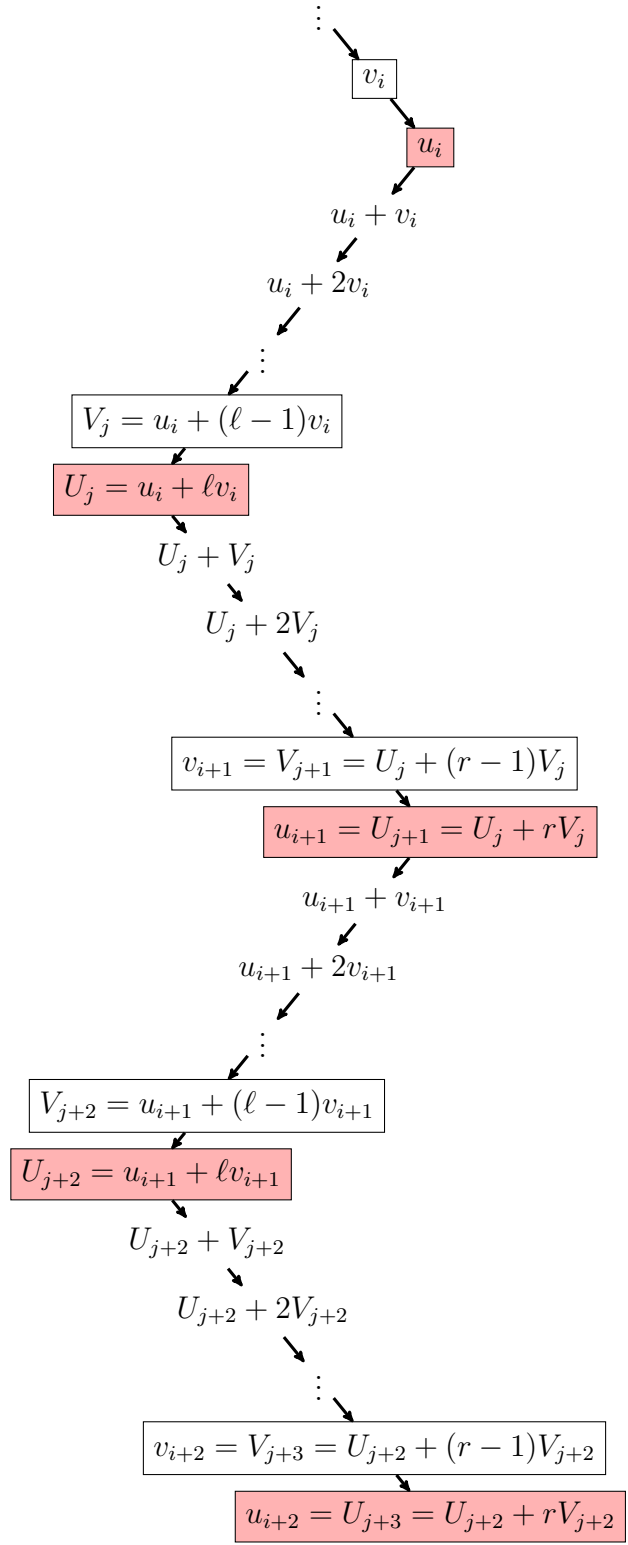

\begin{rem}
Let $\binomh{n}{k}=f_1$ be the initial element, and consider the path $L^\ell R^r$.
Since every second corner element of the path satisfies the recurrence equation~\eqref{seq:-} with $\alpha=\ell r+2$, the number of different paths belonging to different patterns but corresponding to the linear recurrence $(f_i)_{i=1}^\infty$ is the number of the divisors of $\ell r=\alpha-2$.
\end{rem}

Figure~\ref{fig:Pascal_4f-f} gives an example for the case when $\alpha-2=2=2\cdot1=1\cdot2$ and $u_1=f_1=2$, $u_2=f_2=7$. 
Clearly, the patters are $L^2R$ and $LR^2$.

\begin{figure}[h!]
 \centering
 \includegraphics[width=0.67\linewidth]{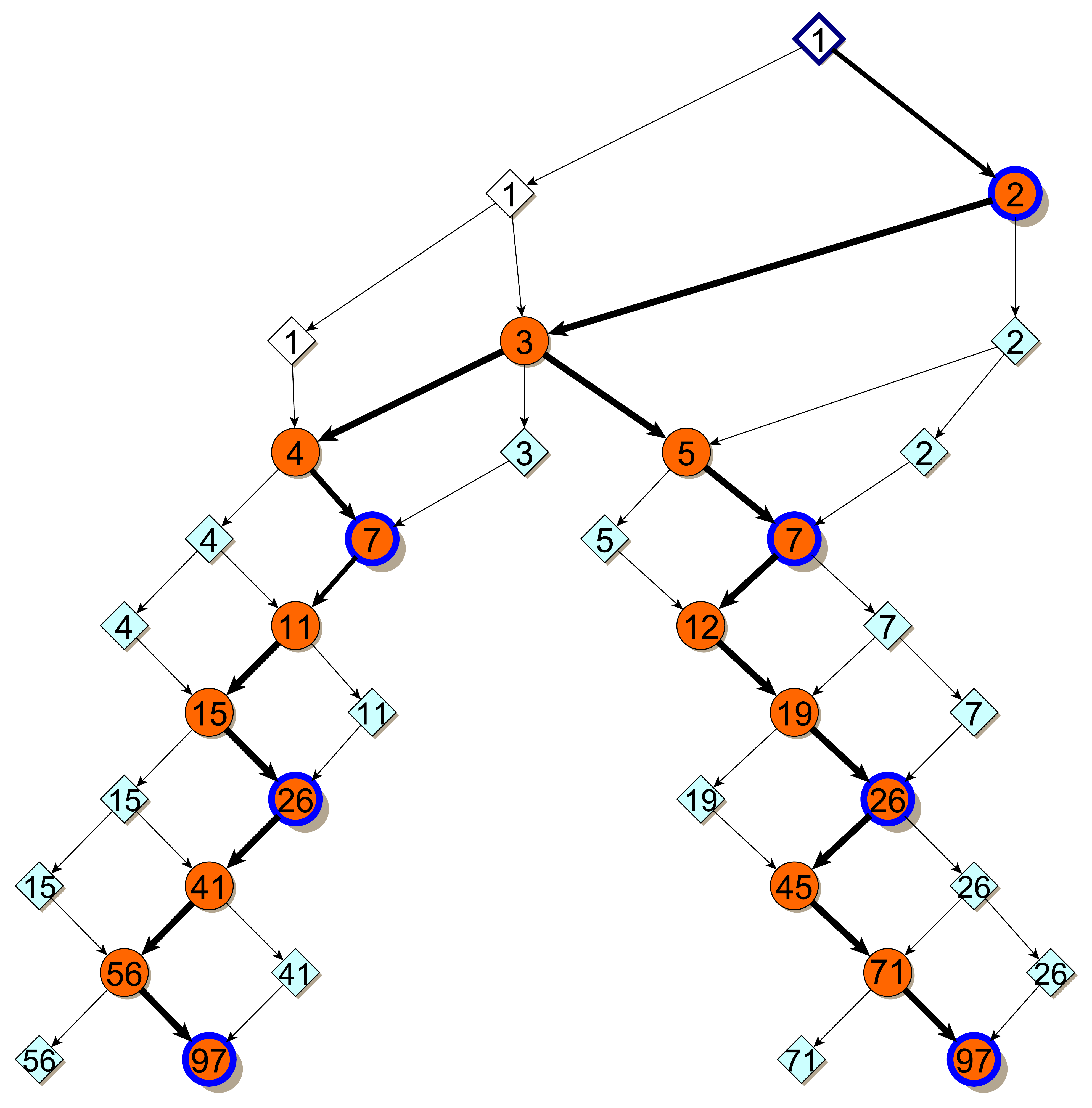}
 \caption{$f_0=1$, $f_1=2$, $f_{i}=4f_{i-1}-f_{i-2}$}
 \label{fig:Pascal_4f-f}
\end{figure}

Now we describe the intermediate sequences located in the path given by $v_1, \binomh{n}{k}=u_1$ and by $L^\ell R^r$. The labels of the elements having distance $(\ell+r)t$ $(t\in\mathbb{N})$ from the base element $\binomh{n}{k}$ are given by suitable sequences $\{w_i\}$.
\begin{theorem}
Put $w_i=u_{i}+mv_i$, where $0\leq m<\ell$, or let $w_i=U_{2i}+mV_{2i}=(m+1)u_{i}+(\ell+m(\ell-1))v_i$, where $0\leq m<r$. Then the terms of the sequence $(w_i)$ satisfy
\begin{equation*} \label{eq:w_lr}
w_{i}=(\ell r+2) w_{i-1}-w_{i-2},\qquad (i\geq3).
\end{equation*}
\end{theorem}

\begin{proof} Consider again Figure \ref{fig:stepsleft-right} to show the statement for the first type of sequences. One can observe the labels of the path described 
by $w_i=u_{i}+mv_i$, where $0\leq m<\ell$ and $i\geq1$.  From \eqref{eq:lr} we see 
\begin{eqnarray*}
w_{i+2}&=&u_{i+2}+m v_{i+2} =(\ell r+2) (u_{i+1}+k v_{i+1})-(u_i+m v_{i})\\
&=&(\ell r+2) w_{i+1}-w_i.
\end{eqnarray*}
The second part of the proof is analoguous. In Figure  \ref{fig:stepsleft-right} the equation $j=2i$ holds, but generally it does not. 
\end{proof}

\begin{cor}\label{cor:W}
In case of $\ell=r$, $W_j=U_j+mV_j$ $(0\leq m<\ell)$ satisfy the equation
\begin{equation}\label{eq:Wl}
W_{j}=\ell W_{j-1}+W_{j-2},\qquad (j\geq3).
\end{equation}
\end{cor}

In Figure~\ref{fig:Pascal_3f+f}, according to Corollary~\ref{cor:W} we give two examples for the representation of elements of recurrence sequence $f_{i}=3f_{i-1}+f_{i-2}$. The pattern of both paths is $R^3L^3$, moreover, $u_1=3$, $v_1=2$, $m=2$ and $u_1=4$, $v_1=3$, $m=1$, respectively. 

\begin{figure}[h!]
 \centering
 \includegraphics[width=0.55\linewidth]{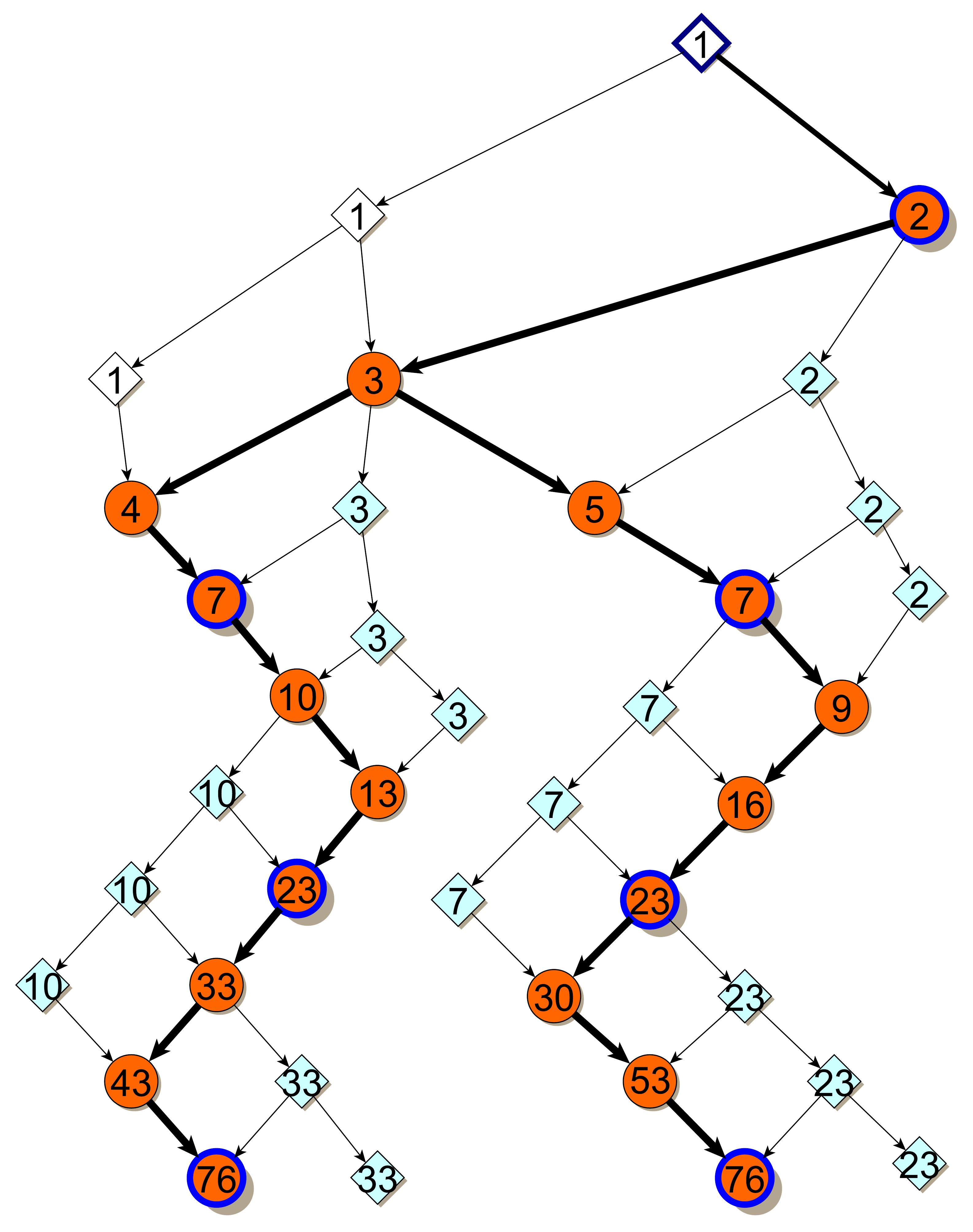}
 \caption{$f_0=1$, $f_1=2$, $f_{i}=3f_{i-1}+f_{i-2}$}
 \label{fig:Pascal_3f+f}
\end{figure}

\begin{theorem}\label{th:back}
Consider the sequence \eqref{seq:-}. If  $\ell r=\alpha-2$, and 
 $$m=\frac{f_{j+1}-(r+1)f_{j}}{\ell+r(\ell-1)}$$
is an integer for some $j\ge1$, further $m<f_j$ holds, then the elements $f_i$ $(i\geq j)$ can be represented in \hpt45 by every second corner elements of a paths given by the the pattern $L^{\ell}R^r$, and by $u_1=f_j$ and $v_1=m$.  
\end{theorem}
\begin{proof}
Let $u_1=f_j$ and $u_2=f_{j+1}$. Then  equation \eqref{eq:ui} yields   $v_1=(f_{j+1}-(r+1)f_j)/(\ell+r(\ell-1))$. Since the integers $v_1$ and $u_1$ are neigbours in a suitable row of \hpt45, therefore there is a path with the pattern $L^{\ell}R^r$ from $v_1$ and $u_1$ such that  every second corner elements are $f_{i+1}$ $(i\geq j)$.  
\end{proof}
\smallskip

Figure~\ref{fig:Pascal_4f-f} gives examples on the paths of $f_{i}=4f_{i-1}-f_{i-2}$ with initial elements $u_1=f_2=7$ and $u_2=f_3=26$,  moreover $v_1=4$ and $v_1=5$, respectively, where $\alpha-2=2=2\cdot1=1\cdot2=lr$, and the patters are $L^2R$ and $LR^2$.

\begin{theorem}
Consider now the sequence \eqref{seq:+}. If  ${\ell}^2=\eta-2$, and $m=(f_{j+1}-f_j)/\ell$ is an 
integer, further $m<f_j$, then the elements $f_i$ $(i\geq j\geq 1)$ can be represented in \hpt45 by every corner elements of the paths given by the pattern $L^{\ell}R^r$, and by $u_1=f_j$ and $v_1=m$.
\end{theorem}
\begin{proof}
The proof is similar to the proof of Theorem \ref{th:back}. 
Using \eqref{eq:Ui}, from $u_1=f_j$ and $u_2=f_{j+1}$   we gain  $v_1=(f_{j+1}-f_j)/{\ell}$.
\end{proof}

\end{document}